\documentclass[10pt, a4paper]{amsart}
\usepackage[english]{babel}
\usepackage{amssymb,amsmath,amsthm}
\usepackage{verbatim}

\numberwithin{equation}{section}

\newtheorem{dummy}{dummy}[section]

\newtheorem{theorem}[dummy]{Theorem}
\newtheorem{corollary}[dummy]{Corollary}
\newtheorem{lemma}[dummy]{Lemma}
\newtheorem{proposition}[dummy]{Proposition}

\def\Co{\mathbb C}
\def\R{\mathbb R}
\def\Q{\mathbb Q}

\def\Z{\mathbb Z}
\def\e{\varepsilon}
\def\t{\theta}
\def\l{\lambda}
\def\={\;=\;}
\def\+{\,+\,}
\def\-{\,-\,}
\def\bal{\begin{aligned}}
\def\eal{\end{aligned}}
\def\be{\begin{equation}\label}
\def\ee{\end{equation}}

\def\L{{\mathcal L}}
\def\re{{\rm Re}}

\def\ch3{\chi_{-3}}
\def\E3{E_{3,\ch3}}
\def\EE3{\widetilde{E}_{3,\ch3}}

\title{Linear Mahler Measures and Double L-values of Modular Forms}
\author{Evgeny Shinder, Masha Vlasenko}

\begin{document}
\maketitle

\begin{abstract}We consider the Mahler measure of the polynomial $1+x_1+x_2+x_3+x_4$, which is the first case not yet evaluated explicitly. 
A conjecture due to F. Rodriguez-Villegas represents this Mahler measure as a special value at the point~4 of the L-function of a modular form of weight~3. 
We prove that this Mahler measure is equal to a linear combination of double L-values of certain meromorphic modular forms of weight 4.
\end{abstract}

\section{Introduction}\label{sec:intro}

The logarithmic Mahler measure of a Laurent polynomial
\[
P(x_1,\dots,x_n) \in \Co[x_1^{\pm 1},\dots,x_n^{\pm 1}]
\]
is defined as
\[
m(P) = \frac{1}{(2\pi i)^n}\int_{|x_1|=\dots=|x_n|=1} \log|P(x_1,\dots,x_n)| \;
\frac{dx_1}{x_1} \dots \frac{dx_n}{x_n}.
\]
One can show that this integral is always convergent. For a monic polynomial in
one variable $P \in \Co[x]$ one can compute $m(P)$ by Jensen's formula
\be{Jensen}
\frac{1}{2\pi i}\int_{|x|=1} \log|P(x)| \; \frac{dx}{x} \= \sum_{\alpha:
P(\alpha)=0} \max(0,\log|\alpha|)\,,
\ee
but no explicit formula is known for polynomials in several variables. Let us
consider the simplest case of linear forms, namely $m(1+x_1+\dots+x_n)$.  In
1981 C.~Smyth discovered~(\cite{Sm1}) that
\be{2var}
m(1+x_1+x_2) \= \frac{3\sqrt{3}}{4\pi}L(\chi_{-3},2)
\ee
where $\chi_{-3}(n)=\bigl(\frac{-3}{n}\bigr)$, $L(\chi_{-3},s) \=
\sum\limits_{n=1}^{\infty} \dfrac{\chi_{-3}(n)}{n^s} \= 1 - \frac1{2^s} +
\frac1{4^s} - \frac1{5^s} + \dots$ and
\be{3var}
m(1+x_1+x_2+x_3) \= \frac{7}{2\pi^2} \zeta(3) \,.
\ee
These formulas can be proved by explicit integration. Later we will see another
method due to F.~Rodriguez-Villegas~(\cite{MMM}) to obtain~\eqref{2var}
and~\eqref{3var} with the help of modular forms. Already in the next case no
explicit formula for $m(1+x_1+x_2+x_3+x_4)$ is known, and this is the subject of
the present paper. One can find in~\cite{VTV} the numerical value 
\[
m(1+x_1+x_2+x_3+x_4) \= 0.544412561752185...
\]
and also there is the following conjectural formula.   

\vskip0.5cm
{\bf Conjecture} (F. Rodriguez-Villegas,~\cite{BLVD}, see also \cite{Zud}):\emph{
\[
 m(1+x_1+x_2+x_3+x_4) \; \overset{?}= \; 6 \, \Bigl(\frac{\sqrt{-15}}{2 \pi
i}\Bigr)^5 L(f_{15},4)
\]
where 
\[
f_{15} \= \eta(3z)^3\eta(5z)^3 + \eta(z)^3\eta(15z)^3 \=  q + q^2 - 3q^3 - 3q^4
+ \dots
\]
is a CM modular form of weight 3, level 15 and Nebentypus $\bigl(
\frac{-15}{\cdot}\bigr)$.} 
\vskip0.5cm

This modular form arises in~\cite{PTV} in relation to the variety
\[
\begin{cases}
1+x_1+x_2+x_3+x_4 \= 0 \\
1+\frac1{x_1}+\frac1{x_2}+\frac1{x_3}+\frac1{x_4} \= 0
\end{cases} 
\]   
which can be compactified to a $K3$ surface of Picard rank 20. Namely, C.~Peters,
J. Top, M. van der Vlugt show in~\cite{PTV} that if $X$ is the minimal resolution of
singularities of the above surface then the L-function of $H^2(X)$ has generic
Euler factor
\[
(1-pT)^{16} \, \Bigl(1 - \bigl( \frac{-3}{p}\bigr) pT\Bigr)^4 \, \Bigl(1 - A_p T
+ \bigl( \frac{-15}{p}\bigr) p^2 T^2 \Bigr) 
\] 
where $A_p$ is the $p$th coefficient in the $q$-expansion of $f_{15}$. 

In order to state our results, consider the modular function $t(z)$ and modular form $f(z)$ of
weight 2 
\be{L3mpar}\bal
t(z) &\= -\Bigl(\frac{\eta(2z) \eta(6z)}{\eta(z) \eta(3z)}\Bigr)^6 \= -q - 6 q^2
- 21 q^3 + \dots\\
f(z) &\= \frac{(\eta(z) \eta(3z))^4}{(\eta(2z) \eta(6z))^2} \= 1 - 4q +  4q^2 -
4 q^3 +\dots
\eal\ee
for the group
\[
\Gamma_0(6)+3 \= \Gamma_0(6) \; \cup \; \Bigl\{ \sqrt{3} \begin{pmatrix} a& b/3
\\ 2 c & d \end{pmatrix} \in {\rm SL}(2,\R) \,|\, a,b,c,d \in \Z \Bigr\} \,.
\]
Throughout the paper we use the differential operator $D=\frac1{2\pi
i}\frac{d}{dz} \= q \frac{d}{dq}$. We need the following modular forms of
weight~4
\be{thm2gs}\bal
g_1 &\= \frac{Dt}{t} f \=  1+2q-14q^2+38q^3-142q^4+252q^5-266q^6+\dots \\
g_2 & \= \frac{t}{1-t} \, g_1 \= -q-7q^2-6q^3+5 q^4+120 q^5 +498 q^6 + \dots \\
g_3 & \= \frac{t(212 t^2 + 251t - 13)}{(1-t)^3} \, g_1 \=
13q+316q^2+2328q^3+\dots \\
\eal\ee
Here $g_1$ is indeed a modular form and one can write it as a linear combination
of Eisenstein series (see~\eqref{thm2g1}), while $g_2$ and $g_3$ have poles at
the discrete set of points where $t(z)=1$. Our main result is the following.

\vskip0.5cm
{\bf Theorem.} \emph{ Consider the Chowla-Selberg period for the field
$K=Q(\sqrt{-15})$ 
\be{ChS15}
\Omega_{15} \= \frac{1}{\sqrt{30\pi}} \bigl( \prod_{j=1}^{14}
\Gamma\bigl(\frac{j}{15}\bigr)^{(\frac{-15}{j})} \bigr)^{1/4} \,, 
\ee
and the two numbers
\be{thm2DLs}
L(g_j,g_1,3,1) \= (2 \pi)^4 \, \int_{0}^{\infty} g_1(i s)
\int_{s}^{\infty}\int_{s1}^{\infty}\int_{s2}^{\infty} g_j(i s_3) \, ds_3 \, ds_2
\, ds_1 \, ds 
\ee
for $j=2,3$. One has
\[\bal
m(1+x_1+x_2+&x_3+x_4) \- \frac45 \; m(1+x_1+x_2+x_3) \\
& \= \frac{3\sqrt{5}\Omega_{15}^2}{20 \pi} L(g_3,g_1,3,1) \- 
\frac{3\sqrt{5}}{10 \pi^3 \Omega_{15}^2} L(g_2,g_1,3,1)\,.
\eal\]}
\vskip0.5cm

The reader will find this statement in a slightly different notation in Corollary~\ref{cor2}. First, let us explain why the integrals
in~\eqref{thm2DLs} converge. For $z \in i \R_+$ both $t(z)$ and $f(z)$ are
real-valued and one can easily check that $t(z)<0$. Therefore $g_2$ and $g_3$
have no poles along the imaginary half-axis. When $s \to \infty$ we have $g_1(i
s) = O(1), g_2(i s) = O(e^{-2 \pi s})$ and $g_3(i s) = O(e^{-2 \pi s})$ because
$q$-expansions of $g_2$ and $g_3$ start in degree~1, therefore the integrated
integrals above are convergent at $\infty$. Also one can show that all three
functions $g_j(i s)$ are $o(s)$ when $s \to 0$, hence they are globally bounded
and there is no problem with convergence at $s=0$. With the help of PARI/GP we
find that numerically
\[\bal
& L(g_2,g_1,3,1) \= -0.44662442...\\
& L(g_3,g_1,3,1) \= 8.5383217...
\eal\]
which agrees with the statement of the theorem.

The geometric meaning of these two numbers is not clear at the moment. If for example $g_2$ were a holomorphic cusp form then the number defined in~\eqref{thm2DLs} would be indeed the value of the corresponding double L-function $L(g_2, g_1, s_2, s_1)$ at $s_2 = 3, s_1=1$, which is the motivation
for our notation. We discuss double L-values of holomorphic modular forms in
Section~\ref{sec:dL}. But as soon as forms under consideration have poles in the
upper half-plane the corresponding multiple integrals become path-dependent and
there is no general theory of multiple L-values. Also we would like to remark
that for the holomorphic modular form $g_1$ in our theorem one has
$m(1+x_1+x_2+x_3)=-\frac12L(g_1,1)$, the reader can find the proof of this
statement in Section~\ref{sec:mp}. Another observation is that the poles of
$g_2$ ant $g_3$ are located at the points from the same field
$K=\Q(\sqrt{-15})$, namely at the images of $z=\frac18+\frac{\sqrt{-15}}{24}$
under the group $\Gamma_0(6)+3$.   

\medskip

The structure of the paper is as follows. 
Sections 2 and 3 follow the approach pioneered by Rodriguez-Villegas \cite{MMM}.
In Section 2 we relate the Mahler measure of $1+x_1+\dots+x_n$ to the
principal period of a pencil of Calabi-Yau varieties of dimension $n-1$ given by
\[
\Bigl(1 + x_1 + \dots + x_n \Bigr)\Bigl( 1 + \frac1{x_1} + \dots +
\frac1{x_n} \Bigr) \= \lambda
\]
and the corresponding Picard-Fuchs differential equation. One needs to do explicitly analytic continuation of its solutions from one singular point to another one in order to compute the Mahler measure. For $n=2,3$ the differential operators appear to have modular parametrization. This allows us to do necessary analytic continuation and derive~\eqref{2var} and~\eqref{3var} in Section~3. 

When $n=4$ the Picard-Fuchs differential operator is not modular. However one can apply Jensen's formula to reduce the number of variables:
in Section~4 we observe that in fact $m(1+x_1 + \dots + x_n)$ can be computed by analytic continuation of a solution of a non-homogeneous differential equation with the Picard-Fuchs differential operator corresponding to $m(1+x_1 + \dots + x_{n-1})$. A non-homogeneous differential equation arises if one considers the generating function for the moments of a solution
of a homogeneous differential equation along a path. Moreover, the differential operator depends only on the initial differential equation being independent of the particular solution and the path, while the right-hand side depends on this data (Proposition~\ref{de_transform}). In Section~6 we discuss a modular interpretation of solutions to a non-homogeneous equation in the case when the differential operator has modular parametrization and show that double L-values of modular forms appear naturally in this context.

Though our main interest is the case $n=4$, we keep applying our technique parallelly to the case $n=3$ throughout the paper (Theorem~\ref{Thm1} and Corollary~\ref{cor1}). 
This leads to a linear relation (\ref{L_relation}) between a double $L$-value of two Eisenstein series of weight~3, and ordinary $L$-values
$L(\chi_3,2)$ and $\zeta(3)$.
We give a direct proof of this relation in Section~7 using a method due to Zudilin \cite{Z2}, \cite{Z3}. 

\medskip

Our original interest in Mahler's measure came from the beautiful paper \cite{MMM}
which has been inspiring us the whole time we were working on this project.
We would like to thank 
our friends and colleagues Sergey Galkin, Vasily Golyshev, Anton Mellit, Maxim Smirnov and Wadim Zudilin 
for their interest in our work.
Both authors are greatful to the Max-Planck-Institute f\"ur Mathematik
in Bonn for providing wonderful working conditions where a significant part of this work has been done. We would like to
express our gratitude to the referee of the manuscript who read it
carefully and helped us to improve the exposition.

\section{Mahler Measures and Differential Equations}\label{sec:mmde}

For a Laurent polynomial $P(x_1,\dots,x_n)$ the function
\[
a(t) \= \frac{1}{(2\pi i)^n}\int_{|x_1|=\dots=|x_n|=1} \frac1{1-t
P(x_1,\dots,x_n)} \; \frac{dx_1}{x_1} \dots \frac{dx_n}{x_n}
\]
is well defined for small $t$ since $|P|$ is bounded on the torus. We call
$a(t)$ the principal period of $P$. It is the generating function for the
sequence 
\be{cterms}
a_m \= \text{ the constant term of } P(x_1,\dots,x_n)^m 
\ee
since
\[\bal
a(t) &\= \sum_{m=0}^{\infty} t^m \; \frac{1}{(2\pi
i)^n}\int_{|x_1|=\dots=|x_n|=1} P(x_1,\dots,x_n)^m \; \frac{dx_1}{x_1} \dots
\frac{dx_n}{x_n} \\
& \= \sum_{m=0}^{\infty} a_m \, t^m\,.
\eal\]

Suppose that the polynomial $P$ takes only nonnegative real values on the torus $\{ |x_i|=1 \}$. Then the 
Mahler measure $m(P)$ can be computed as follows. For small real $t < 0$ one has
\[\bal
m(P-\frac1t) & \= \frac1{(2 \pi i)^n} \int_{|x_i|=1} \log\bigl(P(x_1,\dots,x_n) - \frac 1 t \bigr) \;
\frac{dx_1}{x_1} \dots \frac{dx_n}{x_n} \\
& \= \frac1{(2 \pi i)^n} \int_{|x_i|=1} \log\bigl(- \frac 1 t (1 - t P(x_1,\dots,x_n) )  \bigr) \;
\frac{dx_1}{x_1} \dots \frac{dx_n}{x_n} \\
&\= - \log(-t) - \sum_{m=1}^{\infty} \frac{t^m}{m} \frac1{(2 \pi i)^n}
\int_{|x_i|=1} P(x_1,\dots,x_n)^m \; \frac{dx_1}{x_1} \dots \frac{dx_n}{x_n} \\
&\= - \log(-t) - \sum_{m=1}^{\infty} \frac{t^m}{m} a_m \= - \bigl( t \frac{d}{dt}\bigr)^{-1} a(t)\,.
\eal\]
Though we did the computation only for small real $t<0$, the first integral here and the terminal expression are holomorphic in $t$ and defined in some neighbourhood of the real negative half-axis (apart from possibly finitely many punctures where $a(t)$ has singularities). Therefore  
\be{ancont}
m(P) \= - \re \; \bigl( t \frac{d}{dt}\bigr)^{-1} a(t) \Big|_{t = \infty}\,,
\ee
where the analytic continuation is done along $-\infty < t < 0$ and we added real part to be independent of the branch of $\log(t)$, i.e. we can now assume throughout the paper that 
\[
\bigl( t \frac{d}{dt}\bigr)^{-1}\; \sum_{m=0}^{\infty} a_m t^m \= a_0 \log t \+ \sum_{m=1}^{\infty} \frac{a_m}{m} t^m \,.
\]

On the other hand, it is known (\cite{SB}) that the sequence~\eqref{cterms}
always satisfies a recursion, i.e. $a(t)$ is a solution to an ordinary
differential equation 
\be{de}
\L\Bigl(t, t\frac{d}{dt}\Bigr) a(t) \= 0    
\ee
where $\L$ is a certain polynomial in two non-commuting variables. Finally we
see that the Mahler measure $m(P)$ can be computed by doing analytic continuation of
a particular solution to an ordinary differential equation which one constructs
from the polynomial $P$.  

Let us apply this strategy to the linear polynomials. Observe that 
\[
m(1 + x_1 + \dots + x_n ) \= \frac12 m( P_n ) 
\]
where
\be{Pn}
P_n \= \Bigl(1 + x_1 + \dots + x_n \Bigr)\Bigl( 1 + \frac1{x_1} + \dots +
\frac1{x_n} \Bigr)
\ee
takes nonnegative real values on the torus. Consider the sequence of the constant terms of the powers of $P_n$:
\[\bal
&n=2 \qquad a_m: 1\,,\;3\,,\;15\,,\; 93 \,,\; 639 \; \dots \\
&n=3 \qquad a_m: 1\,,\;4\,,\;28\,,\; 256 \,,\;2716 \; \dots \\
&n=4 \qquad a_m: 1\,,\;5\,,\;45\,,\; 545 \,,\;7885 \; \dots \\
\eal\]

The corresponding differential equations 
\[
\L_n\bigl(t, t \frac{d}{dt}\bigr)\, a(t) \= 0
\]
are given by
\[\bal
&\L_2(t,\t) \= \t^2 \- t(10 \t^2 + 10\t + 3) \+ 9t^2(\t + 1)^2 \\
\\
&\L_3(t,\t) \= \t^3 \- 2t(2 \t + 1)(5 \t^2+5\t+2) \+ 64t^2(\t + 1)^3 \\
\\
&\L_4(t,\t) \= \t^4 \-
t (35 \t^4 + 70 \t^3 + 63 \t^2 + 28\t + 5) \\
&\qquad\qquad\+ t^2 (\t+1)^2 (259\t^2+518\t+285) \- 225 t^3 (\t+1)^2 (\t+2)^2 
\eal\]
(See \cite{Ver1} for the general form of the operator.)
We use the notation $\t = t \frac{d}{dt}$ to distinguish it
from $D = q \frac{q}{dq}$.

In all three cases there is a unique analytic at $t=0$ solution satisfying
$a(t)=1+o(t)$. 

The equations $P_2(x_1,x_2)=\lambda$ and $P_3(x_1,x_2,x_3)=\lambda$ describe families of elliptic curves and $K3$-surfaces
of rank 19 respectively. It is therefore natural that the differential equations $\L_2$ and $\L_3$ have
modular parametrization, in which cases we can easily do analytic continuation of their solutions and compute the corresponding Mahler measures by formula~\eqref{ancont}. We do this in the next section. 

Unfortunately, this method is not applicable in the case $m(1+x_1+x_2+x_3+x_4)$. The equation $P_4(x_1,x_2,x_3,x_4)=\lambda$ describes a family of Calabi-Yau threefolds, and hence we do not expect $\L_4$ to have modular parametrization. Indeed, one can check that the differential
operator $\L_4$ is not a symmetric cube of any second order differential
operator and therefore does not admit a modular parametrization. Later
we show that one can still use the operator $\L_3$ to compute
$m(1+x_1+x_2+x_3+x_4)$, though the price paid is that we have to consider
non-homogeneous differential equations.

\section{Modular parametrizations of $\L_2$ and $\L_3$}\label{sec:mp}

Recall (\cite{Zag}) that for an arbitrary modular function $t(z)$ and a modular form $f(z)$
of weight $k$ on a congruence subgroup of ${\rm SL}(2,\Z)$ one can construct an ordinary differential operator of order $k+1$
with algebraic coefficients
\be{mde}
\sum_{i=0}^{k+1} c_i(t) \Bigl(\frac{d}{dt}\Bigr)^i \,, \quad c_i(t)
\in \overline{\Co(t)}
\ee
such that the functions 
\be{locsys}
f(z), z f(z), \dots , z^{k} f(z)
\ee  
span the kernel of the pull-back of this operator to the upper half-plane
\[
\sum_{i=0}^{k+1} c_i\bigl(t(z)\bigr)
\Bigl(\frac1{t'(z)}\frac{d}{dz}\Bigr)^i \= \sum_{i=0}^{k+1}  \widetilde c_i(z)
\Bigl(\frac{d}{dz}\Bigr)^i\,. 
\]
It follows that an operator with these properties is unique up to multiplication
by algebraic functions of $t$ on the left. On the other hand, the operator
\[
\frac1{t'(z) \cdot f(z)} \Bigl( \frac{d}{dz}\Bigr)^{k+1} \frac1{f(z)} 
\]
obviously annihilates the local system~\eqref{locsys} and it is a routine to
check that if we rewrite it as
\[
\frac1{t'(z) \cdot f(z)} \Bigl( \frac{d}{dz}\Bigr)^{k+1} \frac1{f(z)} \=
\sum_{i=0}^{k+1} g_i(z) \Bigl(\frac{d}{dt}\Bigr)^i 
\]
all coefficients $g_i(z)$ will be modular functions and hence can be written as
some algebraic functions $g_i(z)=c_i\bigl(t(z)\bigr)$. The reader could refer
to~\cite[Proposition 21]{Zag} for several constructions of the differential
equation satisfied by a modular form.

Recall that $D = \frac1{2\pi i}\frac{d}{dz} = q \frac{d}{dq}$. In
view of the above, we make a choice and define the operator
\be{Ltf}
\L_{t,f} \= \frac1{Dt \cdot f} D^{k+1} \frac1{f}\,. 
\ee   
This choice corresponds to the leading coefficient in~\eqref{mde} being 
\[
c_{k+1}(t) \= \frac{D^kt}{f^2}\,.
\]

It is not hard to check that both $\L_2$ and $\L_3$ can be obtained from
certain pairs of a modular function and modular form, namely the following ones
(see~\cite{Verrill} for all the details in the case of $\L_3$).
 
\begin{proposition} With 
\be{L2mpar}\bal
t &\= \frac{\eta(6z)^8 \eta(z)^4}{\eta(3z)^4 \eta(2z)^8} \= q - 4 q^2 + 10 q^3 +
\dots\\
f &\= \frac{\eta(2z)^6 \eta(3z)}{\eta(z)^3 \eta(6z)^2} \= 1 + 3 q + 3 q^2 + 3
q^3 + \dots\\
\eal\ee
one has
\[
\L_{t,f} \= \frac1{t} \L_2\Bigl( t, t \frac{d}{dt} \Bigr)\,.
\]
\end{proposition}

\begin{proposition}\label{L3mparPr} With $t(z)$ and $f(z)$ as in~\eqref{L3mpar}
one has
\[
\L_{t,f} \= \frac1{t} \L_3 \Bigl( t, t \frac{d}{dt} \Bigr)\,.
\]
\end{proposition}

Let us use these modular parametrizations to compute $m(P_2)$ and $m(P_3)$ by
formula~\eqref{ancont}. With~\eqref{L2mpar} we see that at $q=0$ we have $t=0$
and $f=1$, hence $f$ coincides with $a(t)$ near $t=0$. Since $t$ runs over the negative real axis when $z \in \frac12+i\R_{+}$ we have by~\eqref{ancont}
\be{mP2ev}\bal
m(P_2) \= - \re\; \bigl( t \frac{d}{dt}\bigr)^{-1} a(t) \Big|_{t=\infty} & \= -
\re\; \Bigl( \frac{t}{Dt} D\Bigr)^{-1} f \Big|_{z=\frac12} \\
& \= - \re \; D^{-1} \; \Bigl( \frac{Dt}{t} f \Bigr) \Big|_{z=\frac12}\,.
\eal\ee
Here and throughout the paper we make a particular choice for $D^{-1}$ by
letting it to act on $q$-series as
\[
D^{-1} \; \sum_{n=0}^{\infty} a_n q^n \= a_0 \log q \+ \sum_{n=1} \frac{a_n}{n}
q^n \,,
\]
which is in accordance with our choice for $\bigl( t \frac{d}{dt}\bigr)^{-1}$,
so that the above computation is correct. To evaluate the terminal expression in~\eqref{mP2ev} let us recall the definition of L-function of a modular form. 

For a modular form $g=\sum_{n=0}^{\infty} c_n q^n$ of weight $k$ on a congruence subgroup of ${\rm SL}(2,\Z)$ one
has $c_n = O(n^{k-1})$ and the $L$-function of $g$ is defined by
\[
L(g,s) \= \sum_{n=1}^{\infty} \frac{c_n}{n^s}
\]
when $\re \, s > k$. This function can be continued as a meromorphic function to the whole complex plane. Moreover, if $g$ is a cusp form then $L(g,s)$ is  holomorphic everywhere in $\Co$. 

\begin{proposition}\label{singL} Let $g$ be a modular form of weight $k$ with
$c_0=0$ and $p < k$ be an arbitrary integer. If $L(g,s)$ has no poles with $\re s \ge p$ then one has
\[
\underset{q \to 1}\lim \; {\Bigl(D^{-p} g \Bigr)}(q) \= L(g,p)\,.
\]
If $c_0 \ne 0$ the same holds when $p < 0$,
\[
\underset{q \to 1}\lim \; g(q) \= c_0 \+ L(g,0)
\]
and
\[
\underset{q \to 1}\lim \; {\Bigl(D^{-1} g \Bigr)}(q) \= \underset{q \to 1}\lim \;\Bigl(c_0 \log q + \sum_{n=1}^{\infty}\frac{c_n}{n} \; q^n \Bigr) \= L(g,1)
\]
assuming that the branch of $\log q$ is taken so that $\underset{q \to 1}\lim \log q = 0$.
\end{proposition}

The reason we do not consider $p>1$ in the latter case is that it is not that clear how to define $D^{-p}$ when $c_0 \ne 0$. Note also that one always has $L(g,p)=0$ when $p < 0$. Indeed, the function $\Lambda(g,s)=\Gamma(s)/(2\pi)^s L(g,s)$ satisfies a functional equation when $s$ goes to $k-s$ and it is obviously holomorphic when $\re \, s > k$, hence also when $\re \, s < 0$. Since $\Gamma(s)$ has poles at nonpositive integers $L(g,s)$ has zeros at all integers $s<0$. If $g$ is a cusp form then also $L(g,0)=0$ by the same reason since $\Lambda(g,s)$ is holomorphic in the entire complex plane.

\begin{proof}[Proof of Proposition~\ref{singL}] First let $c_0 = 0$ or $c_0 \ne 0$ but $p<0$. When $\re \, w > k-p$ one has
\[\bal
\frac{\Gamma(w)}{(2\pi)^w} L(g,p+w) &\= \frac{\Gamma(w)}{(2\pi)^w}\sum_{n=1}^{\infty}\frac{c_n}{n^{p+w}} \= \sum_{n=1}^{\infty}\frac{c_n}{n^{p}} \int_0^{\infty} t^{w-1} \, e^{-2 \pi n t}\, dt\\
& \= \int_0^{\infty} t^{w-1} \Bigl(D^{-p} g \Bigr) (it) dt \,.
\eal\]
Since $\Gamma(w)$ is small when the imaginary part of $w$ is large and $L(g,p+w)$ is uniformly bounded we can apply the inverse Mellin transform. Namely, with any real $c>k-p$ one has  
\[\bal
\Bigl(D^{-p} g \Bigr) (it) &\= \frac1{2\pi i} \int_{c-i\infty}^{c+i\infty} \frac{\Gamma(w)}{(2\pi t)^w} L(g,p+w) dw \\
&\= L(g,p) \+ \frac1{2\pi i} \int_{-\varepsilon-i\infty}^{-\varepsilon+i\infty} \frac{\Gamma(w)}{(2\pi t)^w} L(g,p+w) dw 
\eal\]
where we moved the path of integration and $0<\varepsilon <1$ is chosen sufficiently small so that $\L(g,s)$ still has no poles with $\re \, s \ge p-\varepsilon$. The last integral is $O(t^{\varepsilon})$ and obviously vanishes when $t\to 0$.

If $c_0 \ne 0$ the calculation above would remain correct if we substitute $D^{-p}(q)$ by $g(q)-c_0$ and $D^{-1}(q)-c_0 \log q$ when $p=0$ and $p=1$ correspondingly.
\end{proof}

Going back to~\eqref{mP2ev}, consider a modular form of weight~3 given by
\be{thm1g1}\bal
g(z) & \= \Bigl( \frac{Dt}{t} f \Bigr)(z+\frac12) \= 1+q-5q^2+q^3+11
q^4-24q^5+\dots \\
& \= E_{3,\chi_{-3}}(z) - 2 E_{3,\chi_{-3}}(2z) - 8 E_{3,\chi_{-3}}(4z) \\
\eal\ee 
where $E_{3,\chi_{-3}} \in M_3(\Gamma_0(3),\chi_{-3})$ is the Eisenstein series
\be{Eis3}
E_{3,\chi_{-3}} \= -\frac19 + \sum_{n \ge 1} \sum_{d|n} \chi_{-3}(d) d^2 q^n \,.
\ee 
Then $L(E_{3,\chi_{-3}},s) \= \zeta(s) \, L(\chi_{-3},s-2)$ and
\[
L(g,s) \= \Bigl( 1 - \frac2{2^s} - \frac{8}{4^s} \Bigr) \zeta(s) \,
L(\chi_{-3},s-2)
\]
is holomorphic in the entire complex plane. Since Fourier coefficients of the form $g$ are real $L(g,s)$ takes real values
at real arguments $s$. Combining~\eqref{mP2ev} with Proposition~\ref{singL} we
get
\[
m(P_2) \= - L(g,1) \= 2 \, \underset{s \to 1}\lim \, \zeta(s) \,
L(\chi_{-3},s-2) \= \frac{3\sqrt{3}}{2 \pi}L(\chi_{-3},2)\,.
\]
Recall that $m(1+x_1+x_2)= \frac12m(P_2)$, hence we have just
reproved~\eqref{2var}.

Analogously, with~\eqref{L3mpar} we have that $t(z)$ assumes all negative real values along the imaginary half-axis and $t=\infty$ at $z=0$, hence
\[
m(P_3) \= - \re \; D^{-1}\, \Bigl( \frac{Dt}{t} f \Bigr) \Big|_{q=1} \,. 
\]
We consider
\be{thm2g1}\bal
g \=  &\frac{Dt}{t} f \= 1+2q-14q^2+38q^3-142q^4+252q^5-266q^6+\dots \\
&\= 2 E_4(z) - 32 E_4(2z) - 18 E_4(3z) + 288 E_4(6z)  
\eal\ee
with the Eisenstein series
\be{Eis4}
E_4 \= \frac1{240} \+ \sum_{n \ge 1} \sum_{d|n} d^3 q^n \,.
\ee
The function $L(E_4,s)=\zeta(s)\zeta(s-3)$ has the only pole at $s=4$, and one can easily see that
\[
L(g,s) \= \Bigl( 2 - \frac{32}{2^s} - \frac{18}{3^s} + \frac{288}{6^s}
\Bigr) \, \zeta(s)\zeta(s-3)
\]
is holomorphic in the entire complex plane because the factor in the brackets vanishes at $s=4$. Finally, 
\[
m(P_3) \= - L(g,1) \= - \Bigl( 2 - \frac{32}{2} - \frac{18}{2} + \frac{288}{6}
\Bigr) \underset{s \to 1}\lim \, \zeta(s)\zeta(s-3) \= \frac{7
\zeta(3)}{\pi^2}\,.
\]
This again reproves~\eqref{3var} because $m(1+x_1+x_2+x_3)= \frac12m(P_3)$.
 
\section{Computation of $m(P_n)$ via $\L_{n-1}$}\label{sec:Ln1}

Observe that in general the Mahler measure of a Laurent polynomial $P$ which takes nonnegative real
values on the torus $\{|x_1|=\dots=|x_n|=1\}$ can be written
as 
\[
m(P) \= \int_{\l_{\min}}^{\l_{\max}} \log(\l) \, a^*(\l) d\l
\]
where
\[\bal
\l_{\min} &\= \underset{|x_1| = \dots = |x_n|= 1}\min \,
P(x_1,\dots,x_n)\,,\\
\l_{\max} &\= \underset{|x_1| = \dots = |x_n|= 1}\max 
P(x_1,\dots,x_n)\,,
\eal\]
and $a^*(\l)$ is equal to the integral over the variety
\[
\{ P(x_1,\dots,x_n) \= \l \} \cap \{ |x_1| = \dots = |x_n|= 1 \} 
\]
of the the $(n-1)$-form $\omega_{\l}$ defined as the residue 
\[
\frac1{(2\pi i)^n} \frac{dx_1}{x_1} \wedge \dots \wedge \frac{dx_n}{x_n} \=
\omega_\l  \wedge d\l\,,
\]
i.e.
\[
a^*(\l) \= \int_{\{ P=\l\} \cap \{|x_i|=1\}} \omega_{\l} \,.
\]

Along the same lines as we did in Section~\ref{sec:mmde}, one can recover $m(P)$
from the generating function of the ``moments'' of $a^*(\l)$
\[
a_m \= \int_{\l_{\min}}^{\l_{\max}} \l^m \; a^*(\l) d\l \,,\qquad a(t) \=
\sum_{m=0}^{\infty} a_m \, t^m 
\]
by formula~\eqref{ancont}. Indeed, the moment $a_m$ is exactly the constant term
of $P^m$, so this $a(t)$ is identical with the one in the previous section. Also
one can see it directly by repeating the old trick in our new notation: for $-\infty < t < 0$
\[\bal
&\int_{\l_{\min}}^{\l_{\max}} \log(\l - \frac 1 t ) \; a^*(\l) d\l  \\
&\= - \log (-t) - \sum_{m=1}^{\infty} \frac{t^m}{m} \int_{\l_{\min}}^{\l_{\max}}
\l^m \; a^*(\l) d\l  \\
&\= - \log(- t) - \sum_{m=1}^{\infty} \frac{t^m}{m} a_m \= - \bigl( t
\frac{d}{dt}\bigr)^{-1} a(t)\,,
\eal\]
hence 
\[
m(P) \= \int_{\l_{\min}}^{\l_{\max}} \log(\l) a^*(\l) d\l \= - \re \; \bigl(
t \frac{d}{dt}\bigr)^{-1} a(t) \Big|_{t = \infty}
\]
where we again assume that the analytic continuation of $a(t)$ is done along the real negative halfaxis. 
In Section~\ref{sec:mmde} we mentioned without a proof that $a(t)$ satisfies a
differential equation~\eqref{de}. Now this fact follows from the proposition
below because $a^*(\l)$ is a period for the 1-parametric family of varieties
$P(x_1,\dots,x_n)=\l$ and it satisfies a certain differential equation 
\[
\widetilde{\L} \bigl(\l, \l\frac{d}{d\l} \bigr) a^*(\l) \= 0\,,
\]
namely the Picard-Fuchs differential equation for this family.
Proposition~\ref{de_transform} states that moments of a solution of a
differential equation satisfy another differential equation determined by the
initial one, though in general they are solutions of this equation with a
right-hand side. This right-hand side is a simple rational function which depends on the path of
integration and on the choice of the solution along this path. In the above
situation with $a^*(\l)$ along the real line from $\l_{\min}$ to $\l_{\max}$
the right-hand side exceptionally appears to vanish. But later we will need a
general case as well.    

For a Laurent series $\sum_n a_n t^n$ we introduce the notations
\[\bal
& \Bigl[ \sum_n a_n t^n \Bigr]_{+} \= \sum_{n \ge 0} a_n t^n \\
& \Bigl[ \sum_n a_n t^n \Bigr]_{-} \= \sum_{n < 0} a_n t^n \\
\eal\]
for its parts with nonnegative and negative powers correspondingly. One then has
\[
\sum_n a_n t^n \= \Bigl[ \sum_n a_n t^n \Bigr]_{-} \+ \Bigl[ \sum_n a_n t^n
\Bigr]_{+}\,.  
\]

\begin{proposition}\label{de_transform}
For a polynomial differential operator
\[
\widetilde{\L}(\l,\t) \= \sum_{i=0}^{M} \sum_{j=0}^{N} c_{ij} \l^i \t^j \,,\quad
\t \= \l \frac{d}{d\l}
\]
consider a solution $F(\l)$ of $\widetilde{\L} F = 0$ along some path between
$\l=\alpha$ and $\l=\beta$. Then for the generating function of its moments
\[
b(t) \= \sum_{n=0}^{\infty} t^n \int_{\alpha}^{\beta} \l^n F(\l) d\l 
\]  
one has
\[
\widetilde{\L}\Bigl(\frac1t,-\theta_t-1\Bigr) b(t)  \= h(t) 
\]
where the right-hand $h(t)$ is a rational function which can have poles at most
at $t=0,\frac1{\alpha},\frac1{\beta}$ and is defined as follows. Let
\[
\widetilde{\L}^{(k)}(\l,\t) \= \sum_{i=0}^M\sum_{j=k}^N c_{ij} \l^i \t^{j-k}
\]
and for given $\l$ consider a rational function of $t$
\[
H_{\l}(t) \= \l \, \sum_{j=0}^{N-1} \bigl(\t^j F \bigr)(\l) \, \Bigl[
\widetilde{\L}^{(j+1)}\Bigl(\frac1t,-\theta_t-1\Bigr) \frac1{1-\l t } \Bigr]_{+}
\,.
\]
Then
\[
h(t) \= \Bigl[ \widetilde{\L}\Bigl(\frac1t,-\t_t-1\Bigr) b(t) \Bigr]_{-} \-
H_{\beta}(t) \+ H_{\alpha}(t) \,. 
\]
\end{proposition}
\begin{proof}
Integration by parts yields
\[
\int_{\alpha}^{\beta} \l^i \bigl(\theta^j F \bigr)(\l) d\l \= \sum_{s=0}^{j-1}
\l^{i+1} (-i-1)^s \theta^{j-1-s} F(\l) \Big|_{\l=\alpha}^{\l=\beta} 
\+ (-i-1)^j \int_{\alpha}^{\beta} \l^i F(\l) d\l \,,
\]
and we apply this formula to every term below to get  
\[\bal
0 & \= \sum_{n=0}^{\infty} t^n \int_{\alpha}^{\beta} \l^n \Bigl(
\widetilde{\L}(\l,\theta_{\l})F \Bigr)(\l) d\l \= \sum_{ij} c_{ij}
\sum_{n=0}^{\infty} t^n \int_{\alpha}^{\beta}  \l^{n+i} \bigl(\theta^j F
\bigr)(\l) d\l \\
& \= \sum_{n=0}^{\infty} t^n \sum_{i,j} c_{ij} \Bigl[ \sum_{s=0}^{j-1}
\l^{n+i+1} (-n-i-1)^s \theta^{j-1-s} F(\l) \Big|_{\l=\alpha}^{\l=\beta} \+ 
(-n-i-1)^j \int_{\alpha}^{\beta} \l^{n+i} F(\l) d\l \Bigr] \\
& \= \sum_{k=0}^{N-1} \theta^k F(\l) \sum_{j \ge k+1} c_{ij}
\frac{\l}{t^i}\sum_{n=0}^{\infty} \l^{n+i} t^{n+i} (-n-i-1)^{j-1-k}
\Big|_{\l=\alpha}^{\l=\beta} \\
& \quad \+  \sum_{i,j} c_{ij} \frac1{t^i} \sum_{n=0}^{\infty} (-n-i-1)^j t^{n+i}
\int_{\alpha}^{\beta} \l^{n+i} F(\l) d\l \\
& \= \sum_{k=0}^{N-1} \theta^k F(\l) \Bigl[
\widetilde{\L}^{(k+1)}\Bigl(\frac1{t},-\theta_t-1\Bigr)\frac1{1-\l t} \Bigr]_{+}
\Big|_{\l=\alpha}^{\l=\beta} \+  \Bigl[
\widetilde{\L}\Bigl(\frac1t,-\theta_t-1\Bigr) b(t) \Bigr]_{+} \\
& \= H_{\beta}(t) \- H_{\alpha}(t) \+ \Bigl[
\widetilde{\L}\Bigl(\frac1t,-\theta_t-1\Bigr) b(t) \Bigr]_{+} \,.
\eal\]
\end{proof}

Now we introduce another idea which will allow us to apply the above
proposition to the case of linear polynomials. We consider $P=P_n$ as defined
in~\eqref{Pn}, and let $\widetilde{\L}_n$ be the corresponding Picard-Fuchs
differential operator. It can be easily recovered from $\L_n$ since (up to a
simple multiplier) the operators $\L_n(t, \t_t)$ and $\widetilde{\L}_n\Bigl(\frac1t,-\theta_t-1\Bigr)$ must be equal. For example, with
\[\bal
\widetilde{\L}_2 &\= 9 \t^2 - \l (10 \t^2 + 10 \t + 3) + \l^2(\t+1)^2 \\
\widetilde{\L}_3 &\= 64 \theta^3 - 2\l (2 \t+1)(5 \t^2 + 5 \t + 2) + \l^2
(\t+1)^3 \\
\eal\]     
one can easily check that
\be{L21t}
\widetilde{\L}_2\Bigl(\frac1t,-\theta-1\Bigr) \= \frac1{9 t^2}
\widetilde{\L}_2(9 t, \theta) \= \frac1{t^2} \L_2(t,\theta)
\ee
and
\be{L31t}
\widetilde{\L}_3\Bigl(\frac1t,-\theta-1\Bigr) \= -\frac1{64 t^2}
\widetilde{\L}_3(64 t, \theta) \= -\frac1{t^2} \L_3(t,\theta)\,.
\ee
Let $a^*(\l)$ be the solution of $\widetilde{\L}_n a^* = 0$ such that
\be{mPn}
m(P_n) \= \int_0^{(n+1)^2} \log(\l) a^*(\l) d\l\,.
\ee
Applying Jensen's formula~\eqref{Jensen} in the variable $x_{n+1}$ gives
\[\bal
&\frac12 m(P_{n+1}) \= \frac1{(2 \pi i)^{n+1}} \int_{|x_i|=1}
\log|1+x_1+\dots+x_{n+1}| \; \frac{dx_1}{x_1} \dots \frac{dx_{n+1}}{x_{n+1}} \\
&\= \frac1{(2 \pi i)^{n}} \int_{|x_i|=1, |1+x_1+\dots+x_n|>1}
\log|1+x_1+\dots+x_n| \; \frac{dx_1}{x_1} \dots \frac{dx_n}{x_n}
\eal\]
or
\be{mPn1}
m(P_{n+1}) \= \int_1^{(n+1)^2} \log(\l) a^*(\l) d\l\,.
\ee
Observe that~\eqref{mPn} and~\eqref{mPn1} differ only by the lower limit of
integration. This approach allows us to state that for every $n$ there is a
simple rational function $h_n(t)$ and an analytic solution $b_n(t)$ of 
\[
\L_n\Bigl(t, t \frac{d}{dt}\Bigr) b_n(t) \= h_n(t)
\]
such that 
\[
m(P_{n+1}) \= - \re \; \bigl( t \frac{d}{dt}\bigr)^{-1} b_n(t) \Big|_{t =
\infty}\,.
\]
Below we give exact statements for $n=2$ and $n=3$. We will formulate our results
using solutions with rational coefficients rather then their
transcendental linear combinations. This makes sense from the number-theoretical
point of view and will be used later.    

\bigskip

\begin{theorem}\label{Thm1} Take $\L_2(t, \t) \= \t^2 \- t (10 \t^2 + 10 \t + 3)
\+ 9 t^2 (\t+1)^2$ and consider the following analytic at $t=0$ solutions of 
\[\bal
& \L_{2}\bigl(t, t \frac{d}{dt}\bigr)\, \phi(t) \= 0 \qquad \phi(t)\=1+3t+\dots
\\
& \L_{2}\bigl(t, t \frac{d}{dt}\bigr)\, \psi(t) \= \frac{t}{1-t} \qquad \psi(t)
\= t + \dots
\eal\]
Then
\[
m(P_3) \= - \re \, \bigl( t \frac{d}{dt}\bigr)^{-1} [\frac34 \phi(t) +
\frac{6}{\pi^2} \psi(t)] \Big|_{t = \infty} \,.
\]
\end{theorem}

\bigskip

Though all Mahler measures in this theorem were already computed before, we will
use our result to relate them to double L-values of modular forms later. The next
theorem already deals with the interesting case $n=4$.   

\bigskip

\begin{theorem}\label{Thm2} Let $\Omega_{15}$ be the Chowla-Selberg period for
the field $K=Q(\sqrt{-15})$ as in~\eqref{ChS15} and $b(t)$ be the unique
analytic at $t=0$ solution of the non-homogeneous differential equation 
\[
\L_{3}\bigl(t, t \frac{d}{dt}\bigr)\, b(t) \= -\frac{3\sqrt{5}\Omega_{15}^2}{10
\pi}\frac{t(212 t^2 + 251t - 13)}{(1-t)^3} \+ \frac{3\sqrt{5}}{5 \pi^3
\Omega_{15}^2} \frac{t}{1-t}\,.
\]
satisfying $b(t) \= \frac45 + O(t)$. Then
\[
m(P_4) \= - \re \; \bigl( t \frac{d}{dt}\bigr)^{-1} \, b(t) \Big|_{t =
\infty}\,.
\]
\end{theorem}


\section{Proofs}

Recall from Section~\ref{sec:mmde}
that the function
\[
a(t) \= \frac{1}{(2\pi i)^n}\int_{|x_1|=\dots=|x_n|=1} \frac1{1-t
P_n(x_1,\dots,x_n)} \; \frac{dx_1}{x_1} \dots \frac{dx_n}{x_n}
\]
is the unique analytic at $t=0$ solution of $\L_n a = 0$ satisfying $a(0)=1$.
Let us write $a(t)$ as    
\[
a(t) \= \int_0^{(n+1)^2} \frac{a^*(\l)}{1-t \l} d\l\,,
\]
where $a^*(\l)$ is a solution of $\widetilde{\L}_n a^* = 0$. 
Let us also introduce
\[\bal
b(t) &\= \int_1^{(n+1)^2} \frac{a^*(\l)}{1-t \l} d\l \= b_0 + b_1 t + \dots \\
c(t) & \= \int_0^{1} \frac{a^*(\l)}{1-t \l} d\l \= c_0 + c_1 t + \dots \\
\eal\]

In Section~\ref{sec:Ln1} we showed that
\[
m(P_{n+1}) \= - \re \; \bigl( t \frac{d}{dt}\bigr)^{-1} b(t) \Big|_{t = \infty}.
\]

In this section we will compute the coefficient $b_0$ and the rational function 
\[
h(t) := \L_n b(t) = -\L_n c(t)\,,
\]
for $n=2,3$. The proofs of Theorems~\ref{Thm1} and~\ref{Thm2} go along the same lines. First, we identify the solution $a^*(\lambda)$ for $\lambda \in [0,1]$ in terms of the Frobenius basis in the space of solutions near $\lambda=0$. In order to do this we use asymptotics of $a(t)$ when $t$ is large, and we use the modular parametrization of the differential equation to find this asymptotics. As soon as we know $a^*(\lambda)$ explicitly, we apply Proposition~\ref{de_transform} to finish the proof. 

We note that the differential operator $\widetilde{\L}_2$ has singularities at $\l=0,1$ which are regular singular points
of maximal unipotent monodromy, whereas the operator $\widetilde{\L}_3$ has a regular singular point of
maximal unipotent monodromy at $\l=0$ and is nonsingular at $\l=1$. The period $\Omega_{15}$ appears in Theorem~\ref{Thm2} because the point $\l=1$ corresponds under the modular
parametrization of $\widetilde{\L}_3$ to a CM point of conductor $15$.

It hopefully will not confuse the reader that we use the same notation 
\[
a(t), a^*(\l), b(t), c(t)
\]
for the integrals corresponding to the case $n=2$ in Lemmas $5.1-5.3$ and to the case $n=3$ in Lemmas $5.4-5.7$.

\begin{proof}[Proof of Theorem~\ref{Thm1}] 
Here $a(t) = \phi(t)$.
From Proposition~\ref{de_transform} and Lemmas~\ref{rhs1_2} and \ref{rhs2_2} below it follows that 
\[
\widetilde{\L_2}\Bigl(\frac1t,-\t-1 \Bigr) c(t) \= -\frac{6}{\pi^2} \Bigl( \frac1{1-t} +\frac{1}{t}\Bigr)
\= -\frac{6}{\pi^2} \frac1{t(1-t)}.
\]

According to~\eqref{L21t} we then have
\[
\L_2(t,\t) c(t) \= t^2 \widetilde{\L_2}\Bigl(\frac1t,-\t-1 \Bigr) c(t) \= -\frac{6}{\pi^2}\frac{t}{1-t} \, 
\]
and then
\[
\L_2(t,\t) b(t) \= -\L_2(t,\t) c(t) = \frac{6}{\pi^2}\frac{t}{1-t}. 
\]

From Lemma~\ref{rhs2_2} we have $b_0 \= 1 - c_0 \= \frac34$
and therefore $b(t) = \frac34 \phi(t) + \frac{6}{\pi^2} \psi(t)$.
To finish the proof it remains to verify the three lemmas that follow below.

\end{proof}



\begin{lemma} As  $\l \to 0$, we have
\label{asymp_2}
\[
a^*(\l)  \= \frac1{\sqrt{3} \pi} + O(\l)\,. 
\]
As  $\l \to 1^-$, we have
\[
a^*(\l)  \= -\frac{3}{4 \pi^2} \log(1-\l) + O(1)\,.
\]
\end{lemma}

\begin{proof}
Using the modular parametrization~\eqref{L2mpar} we find that when $t \to -\infty$ along the negative real axis (this corresponds to $z$ going down the ray $\frac12 + i \R_{+}$)
\be{asa1}
t \, a(t) \= \frac1{\sqrt{3} \pi} \log\Bigl( -\frac1{t} \Bigr) \+ O(1) \,. 
\ee 
On the other hand 
\be{asa2}
t \, a(t) \= \int_0^{9} \frac{a^*(\l)}{1/t - \l} d\l \= - \int_0^{9} \frac{a^*(\l)}{s + \l} d\l \Big|_{s = -\frac1t}\,.
\ee
Let us write $a^*(\l) \= \alpha_0 \phi_0(\l) \+ \alpha_1 \phi_1(\l)$ in terms of the Frobenius basis 
\[\bal
&\phi_0(\l) \= 1 + O(\l)\,, \\
&\phi_1(\l) \= \log(\l) \phi_0(\l) + O(\l) \\
\eal\]
of solutions near $\l=0$. One can easily check that for any $\varepsilon>0$
\[\bal
 \int_0^{\e}\frac{d\l}{s+\l} &\= - \log s + O(1) \,,\\
 \int_0^{\e}\frac{\log{\l} d\l}{s+\l} &\= -\frac12 (\log s)^2 + O(1)\\
\eal\]
as $s \to 0$. Comparing~\eqref{asa1} and~\eqref{asa2}, we see that $\alpha_1=0$ and $\alpha_0=\frac1{\sqrt{3} \pi}$. 

Now let
\[\bal
&\kappa_0(\l) \= 1 + O(\l-1)\,, \\
&\kappa_1(\l) \= \log(\l-1) \kappa_0(\l) + O(\l-1) \\
\eal\] be the Frobenius basis at $\l=1$
and $a^*(\l)= \alpha_0 \kappa_0(\l) + \alpha_1 \kappa_1(\l)$ when $\l \to 1_{-}$. 
Using our modular parametrization we find that 
\[
\alpha_1 \= \underset{\l \to 1_{-}} \lim \frac{a^*(\l)}{\log(1-\l)} \= \underset{s \to 0_{+}} 
\lim \frac1{\sqrt{3}\pi} \frac{f(z)}{\log(1-9 t(z))} \Big|_{z=i s} \= - \frac{3}{4 \pi^2} \,.
\]

\end{proof}

\begin{lemma} 
\label{rhs1_2}
In the notation of Proposition~\ref{de_transform} applied to 
$\widetilde{\L}=\widetilde{\L}_2,\; \alpha=0, \; \beta=1,\; F(\l)=a^*(\l)$
we have: 
\[\bal
H_0(t) & \= 0\,, \\
H_1(t) & \= \frac{6}{\pi^2 \, (1-t)}\,.
\eal\]
\end{lemma}
\begin{proof}
In the course of the proof we rely on the asymptotics given in Lemma~\ref{asymp_2}.
Since $\l=0,1$ are singular points we are going to compute $H_0(t)$ and $H_1(t)$
as the corresponding limits of $H_{\l}(t)$. We have
\[\bal
\Bigl[ \widetilde{\L}_2^{(2)}\Bigl(\frac1t,-\theta-1\Bigr) \frac1{1- \l t}
\Bigr]_{+} &\= \Bigl[ \Bigl( 9 -\frac{10}t+\frac1{t^2} \Bigr) \frac1{1- \l t}
\Bigr]_{+} \\
&\= \frac{9 t^2-10t+1}{t^2(1- \l t)} - \frac1{t^2} - \frac{\l-10}{t} \=
\frac{(\l-1)(\l-9)}{1-\l t} \\
\eal\]
and
\[\bal
\Bigl[ \widetilde{\L}_2^{(1)} & \Bigl(\frac1t,-\theta-1\Bigr) \frac1{1- \l t}
\Bigr]_{+} \= \Bigl[ \Bigl( \bigl(9 -\frac{10}t+\frac1{t^2} \bigr)(-\t-1) +
\bigl(-\frac{10}t+2{t^2}\bigr) \Bigr) \frac1{1- \l t} \Bigr]_{+} \\
&\= - \bigl(9 -\frac{10}t+\frac1{t^2} \bigr)\frac1{(1-\l t)^2} +
\bigl(-\frac{10}t+2{t^2}\bigr) \frac1{1- \l t} - \frac1{t^2} \=
-\frac{(\l-1)(\l-9)}{(1-\l t)^2} \,.
\eal\]
These functions have finite limits when $\l \to 0$ and now we see that
$H_0(t)=0$ because $a^*(\l)$ is analytic at $\l=0$ and therefore 
$\underset{\l\to 0}\lim \l \, \theta^j a^*(\l) = 0$ for any $j \ge 0$. 

Since $\underset{\l\to 1_{-}}\lim (\l-1) \, a^*(\l) = 0$ and 
$\underset{\l \to 1_{-}}\lim (\l-1) \,\t a^*(\l) = -\frac{3}{4 \pi^2}$ we find that $H_1(t)=\frac{6}{\pi^2 \, (1-t)}$. 
\end{proof}

\medskip

\begin{lemma}
\label{rhs2_2}
The first coefficient $c_0$ in the power series expansion of $c(t)$ is equal to $\frac14$ and we have
\[
\Bigl[\widetilde{\L}_2\Bigl(\frac1t,-\theta-1\Bigr) c(t) \Bigr]_{-} \= -\frac{6}{\pi^2 \,t}.
\]
\end{lemma}
\begin{proof}
It is easy to compute that
\[
\Bigl[\widetilde{\L}_2\Bigl(\frac1t,-\theta-1\Bigr) c(t) \Bigr]_{-} \= \frac{-3
c_0 + c_1}{t}\,.
\]
Using the modular parametrization~\eqref{L2mpar} (with modular $t$ and $f$
from~\eqref{L2mpar} and $\l = 9 t$ one has $\L_{\l, f}=\frac1{9
\l}\widetilde{\L}_2(\l,\l \frac{d}{d\l})$) we compute that
\[\bal
c_0 \= \int_{0}^{1} a^*(\l) d\l & \= \frac{9}{\sqrt{3} \pi} \int_{i \infty}^0
f(z) t'(z) dz \\
& \=  \frac{18}{\sqrt{3}} \int_{0}^{\infty} f(z)^3 t(z)(1-9 t(z))(1-t(z))
\Big|_{z=is} \, ds \= \frac14\\
\eal\]
and
\[\bal
c_1 - 3 c_0 & \=  \int_{0}^{1} (\l-3) a^*(\l) d\l \\
& \= \frac{18}{\sqrt{3}} \int_{0}^{\infty} (9 t(z)-3) f(z)^3 t(z)(1-9
t(z))(1-t(z)) \Big|_{z=is} \,  ds\\
& \= - \frac{6}{\pi^2}\,.
\eal\]
  
\end{proof}

\begin{proof}[Proof of Theorem~\ref{Thm2}]
We let $\Omega = \Omega_{15}$. 
According to Proposition~\ref{de_transform} and Lemmas~\ref{rhs1_3} and \ref{rhs2_3} below we
compute that
\[\bal
\widetilde{\L_3}\Bigl(\frac1t,-\t-1 \Bigr) c(t) & \= 
\frac{3\sqrt{5}\Omega^2}{10\pi} \Bigl( \frac{13}{t} - \frac{-13 t^2 + 251t + 212}{(1-t)^3}\Bigr) 
\+ \frac{3\sqrt{5}}{5 \pi^3 \Omega^2} \Bigl( \frac1{t} + \frac1{1-t} \Bigr) \\
& \= \frac{3\sqrt{5}\Omega^2}{10\pi}\frac{-212 t^2 - 251t + 13}{t(1-t)^3} \+ \frac{3\sqrt{5}}{5 \pi^3 \Omega^2} \frac1{t(1-t)}.
\eal\]

Now it follows from~\eqref{L31t} that we have
\[\bal
-h(t) = \L_3(t,\t) c(t) & \= -t^2 \widetilde{\L_3}\Bigl(\frac1t,-\t-1 \Bigr) c(t) \\
& \= \frac{3\sqrt{5}\Omega^2}{10 \pi}\frac{t(212 t^2 + 251t - 13)}{(1-t)^3} \-
\frac{3\sqrt{5}}{5 \pi^3 \Omega^2} \frac{t}{1-t}. 
\eal\]

Therefore the function
\[
b(t) \= \int_1^{16} \frac{a^*(\l)}{1-t \l} d\l
\]
satisfies $\L_3 b = h(t)$ and its power series expansion at $t=0$ starts with $b_0 \= a_0 -
c_0 \= \frac45$.
The proof will be finished after verifying the three lemmas below.
\end{proof}

\begin{lemma} 
\label{asymp_3}
In terms of the Frobenius basis
\[\bal
&\phi_0(\l) \= 1 + O(\l) \,, \\
&\phi_1(\l) \= \log(\l) \phi_0(\l) + O(\l) \= \log(\l) + o(1) \,,\\
&\phi_2(\l) \= \log(\l)^2 \phi_0(\l) + O(\l) \= \log(\l)^2 + o(1)  \\
\eal\]
of solutions near $\l=0$ we have
\[
a^*(\l) \= -\frac3{8 \pi^2} \bigl( \phi_1(\l) \- 6 \log 2 \,\phi_0(\l) \bigr)
\,.
\]

At $\l = 1$ we have
\[\bal
a^*(1) &\=  0.1649669005300320... \= \frac{3 \sqrt{5}}{2 \pi}\Omega^2 \qquad
\,,\\
\t a^*(1) &\= -0.032993380106006... \= -\frac{3\sqrt{5}}{10\pi}\Omega^2 \,,\\
\t^2 a^*(1) &\= 0.00330836512971504... \= \frac{\sqrt{5}}{150} \Bigl( \frac{13
\Omega^2}{\pi} - \frac2{\pi^3 \Omega^2} \Bigr).\\
\eal\]
 
\end{lemma}
\begin{proof}
With the help of modular parametrization~\eqref{L3mpar} we
find that when $t \to -\infty$ along negative real axis (this corresponds to $z$
going down to $0$ along the imaginary axis)
\[
t a(t) \+ \frac{3}{16 \pi^2} \log\Bigl( -\frac1{64 t} \Bigr)^2 \to 0\,.
\] 
On the other hand 
\[
t \, a(t) \= \int_0^{16} \frac{a^*(\l)}{1/t - \l} d\l \= - \int_0^{16}
\frac{a^*(\l)}{s + \l} d\l \Big|_{s = -\frac1t}
\]
and since for any $\e > 0$ one has when $s \to 0$
\[\bal
 \int_0^{\e}\frac{d\l}{s+\l} &\= - \log s + O(1)\,, \\
 \int_0^{\e}\frac{\log{\l} d\l}{s+\l} &\= -\frac12 (\log s)^2 + O(1) \,,\\
 \int_0^{\e}\frac{(\log{\l})^2 d\l}{s+\l} &\= -\frac13 (\log s)^3 +
O\bigl(\log(s)^2\bigr) \,,\\
\eal\]
we find that $\alpha_2=0$, $\alpha_1=-\dfrac{3}{8 \pi^2}$ and
$\alpha_0=\dfrac{9}{4 \pi^2} \log 2$.

\medskip


We indicate how to find the values $\t^j a^*(1)$ for $j=0,1,2$. 
With modular $t$ and $f$
from~\eqref{L3mpar} and $\l = 64 t$ one has $\L_{\l, f}=\frac1{64
\l}\widetilde{\L}_3(\l,\l \frac{d}{d\l})$. This $\lambda(z)$ takes real values
from the interval $(0,1]$ for $z \= \frac12 + i s$ and $s \in
(+\infty,\frac{\sqrt{15}}{6}]$. In particular, $\lambda(\tau)=1$ for $\tau = \frac12+\frac{\sqrt{-15}}6$. Using asymptotics at
$\infty$ one can check that on the vertical half-line from $\tau$ to $\infty$ 
\[
a^*\Bigl((\lambda(z)  \Bigr) \=  -\frac3{8 \pi^2} \cdot 2 \pi i \bigl(z -
\frac12 \bigr) \, f(z) \,.
\]
Now the problem is reduced to computing the values of modular forms and their
derivatives at a CM-point of conductor $15$, this leading to expressions involving
$\Omega$ and $\pi$, see \cite[Propositions~26,~27 and Corollary of Proposition~27]{Zag}.
 
\end{proof}

\begin{lemma} 
\label{rhs1_3}
In the notation of Proposition~\ref{de_transform} applied to 
$\widetilde{\L}=\widetilde{\L}_3,\; \alpha=0, \; \beta=1,\; F(\l)=a^*(\l)$
we have: 
\[\bal
H_0(t) & \= 0 \\
H_1(t) & \= \frac{3 \Omega^2\sqrt{5}}{10\pi} \frac{-13t^2+251t+212}{(1-t)^3}- \frac{3\sqrt{5}}{5\pi^3\Omega^2} \frac{1}{1-t}
\eal\]
\end{lemma}
\begin{proof}

One easily checks that $\underset{\l \to 0}\lim \l \, \theta^j a^*(\l) = 0$,
whence $H_0(t)=0$. In order to compute $H_1(t)$ we need
\[\bal
&\Bigl[ \widetilde{\L}_3^{(3)}\Bigl(\frac1t,-\theta_t-1\Bigr) \frac1{1-t}
\Bigr]_{+} \= \frac{45}{1-t}\\ 
&\Bigl[ \widetilde{\L}_3^{(2)}\Bigl(\frac1t,-\theta_t-1\Bigr) \frac1{1-t}
\Bigr]_{+} \= \frac{9(t-6)}{(1-t)^2}\\
&\Bigl[ \widetilde{\L}_3^{(1)}\Bigl(\frac1t,-\theta_t-1\Bigr) \frac1{1-t}
\Bigr]_{+} \= \frac{29+68 t-7t^2}{(1-t)^3},\\
\eal\]
and then the formula for $H_1(t)$ follows after a simple computation with
values of $\t^j a^*(1)$ provided by Lemma~\ref{asymp_3}.
\end{proof}

\begin{lemma}
\label{rhs2_3}
The first coefficient $c_0$ in the power series expansion of $c(t)$ is equal to $\frac15$ and we have
\[
\Bigl[\widetilde{\L}_3\Bigl(\frac1t,-\theta-1\Bigr) c(t) \Bigr]_{-} \= 
\Bigl(\frac{39\sqrt{5}}{10\pi}\Omega^2 + \frac{3\sqrt{5}}{5\pi^3\Omega^2}\Bigr) \frac1{t}.
\]
\end{lemma}
\begin{proof}
It is easy to compute that
\[
\Bigl[\widetilde{\L}_3\Bigl(\frac1t,-\theta-1\Bigr) c(t) \Bigr]_{-} \= \frac{4
c_0 - c_1}{t}\,.
\]
We have
\[\bal
c_0 &\= \int_0^1 a^*(\l) d\l \= -\frac3{8 \pi^2} \cdot 2 \pi i \int_{i
\infty}^{\tau} \bigl(z - \frac12 \bigr) \, f(z) \, \l'(z) dz \\
& \bigl(\text{ here we use that } \, \l = 64 t \, \text{ and } \bigl(q
\dfrac{dt}{dq}\bigr)^2 / f^2 \= t^2 (1-4 t)(1 - 16t ) \quad \bigr) \\
&\=  -\frac3{8 \pi^2} \cdot (2 \pi i)^2 \cdot 64  \int_{i \infty}^{\tau} \bigl(z
- \frac12 \bigr) \, f(z)^2 t(z) \sqrt{(1 - 4 \, t(z))(1-16 \, t(z))} dz \\
&\=  96  \int_{\frac{\sqrt{15}}{6}}^{\infty} s \, f(z)^2 t(z) \sqrt{(1 - 4 \,
t(z))(1-16 \, t(z))} \Big|_{z = \frac12+is} \, ds \= \frac 15   
\eal\]
and
\[\bal
c_1 &- 4 c_0 \= \int_0^1 (\l - 4) a^*(\l) d\l \\ 
&\=  96  \int_{\frac{\sqrt{15}}{6}}^{\infty} s \, f(z)^2 (64 \, t(z) - 4) t(z)
\sqrt{(1 - 4 \, t(z))(1-16 \, t(z))} \Big|_{z = \frac12+is} \, ds \\
&\= -0.708951451918989714... \= - 7 \, a^*(1) - 9 \, \t a^*(1) + 45 \, \t^2a^*(1) \\
&\= -\frac{39\sqrt{5}}{10\pi}\Omega^2 - \frac{3\sqrt{5}}{5\pi^3\Omega^2}
\eal\]
\end{proof}

\section{Double L-values of modular forms}\label{sec:dL}

In Theorem~\ref{Thm1} we are led to the evaluation of
\[
\bigl( t \frac{d}{dt}\bigr)^{-1} \psi(t) \Big|_{t = \infty}
\]  
where $\psi$ is the unique analytic at $t=0$ solution of
\[
\L_{2}\bigl(t, t \frac{d}{dt}\bigr)\, \psi(t) \= \frac{t}{1-t}
\]
which satisfies the condition $\psi(t)=t+o(t)$. The same happens in
Theorem~\ref{Thm2}. Putting this situation into a more general context, consider
a solution of a non-homogeneous differential equation of order $k+1$ which has a
modular parametrization, i.e.
\[
\L_{t,f} \; \psi \= h(t) \,,
\]
where $t$ is a modular function, $f$ is a modular form of weight $k$, $\L_{t,f}$
is defined by~\eqref{Ltf} and $h(t)$ is a function of $t$ which will be just
rational in our cases. We can consider $\psi=\psi(t(z))$ as a function in the
upper half-plane and we rewrite the above differential equation as
\[
\frac1{Dt \cdot f} D^{k+1} \frac{\psi}{f} \= h(t) \,,
\] 
or
\[
D^{k+1} \frac{\psi}{f} \= h(t) \cdot Dt \cdot f \,.
\] 
Therefore $\dfrac{\psi}f$ is an Eichler integral of the modular form $h(t) \cdot
Dt \cdot f$ of weight $k+2$. (This conclusion is precisely the statement of
Lemma~1 in~\cite{Yang}.) Let us assume in addition that the modular function $t$
takes values $0$ and $\infty$ at $q=0$ and $q=1$ correspondingly. Then
\be{dL_first}
\bigl( t \frac{d}{dt}\bigr)^{-1} \psi(t) \Big|_{t = \infty} \= D^{-1} \Bigl(
\frac{Dt}{t} \psi \Bigr)\, \Big|_{q=1} \= D^{-1} \Bigl( g_2 \, D^{-k-1} g_1
\Bigr) \, \Big|_{q=1} 
\ee
where
\[\bal
g_1 \= h(t) \cdot Dt \cdot f \,, \qquad g_2 = \frac{Dt}{t} f  
\eal\]
are two modular forms of weight $k+2$. According to the proposition below, the
right-hand side of~\eqref{dL_first} appears to be a double L-value of these two forms. Let us
give the definition.

Let $g_1=\sum_{n \ge 0} a_n q^n$ and $g_2=\sum_{m \ge 0} b_m q^m$ be two modular forms of weight $k$ on a congruence subgroup of ${\rm SL}(2,\Z)$, and let in addition $a_0=0$. Their double L-function (denoted by $L^{\bullet}$ in \cite{Ramesh}) is defined for $\re(s_1+s_2)>2k$, $\re \, s_2 >k$
by
\[
L(g_1,g_2,s_1,s_2) \= \sum_{n=1}^{\infty} \sum_{m=0}^{\infty} \frac{a_n
b_m}{n^{s_1}(n+m)^{s_2}}\,.
\]
The question of the analytic continuation simultaneously in the two variables $s_1,s_2$ is
rather tricky and we do not want to consider it here. But it appears that for
any fixed integer $s_1=p$ the function $L(g_1,g_2,p,s_2)$ is well-defined for
$s_2$ with sufficiently large real part, and we can easily prove analytic
continuation in this variable when $p>0$. In order to do this one writes
(\cite{Ramesh}) 
\be{2.3}
L(g_1,g_2,p,s_2) \= \frac{(2\pi)^{p+s_2}}{\Gamma(p)\Gamma(s_2)} \sum_{m=0}^{p-1}
\Lambda(g_1,g_2,p-m,s_2+m)
\ee 
where 
\[
\Lambda(g_1,g_2,s_1,s_2) \= \int_0^{\infty} t^{s_2-1} g_2(i t) \int_{t}^{\infty}
v^{s_1-1} g_1(i v) dv \, dt \,.
\]  
Now observe that these integrals are well defined for all $s_1,s_2$. Indeed, this
follows from the estimates
\[\bal
\int_{t}^{\infty} v^{s_1-1} g_1(i v) dv &\= O(t^{s_1-1} e^{-2 \pi t})\,, \quad t
\to \infty \\
&\= O(t^{s_1-k} e^{- \frac{2 \pi}t })\,, \quad t \to 0\\
\eal\]
since $g_1(it) = O \Bigl( t^{-k} e^{- \frac{2 \pi}t } \Bigr)$ when $t \to 0$.
Therefore formula~\eqref{2.3} gives the analytic continuation of~$L(g_1,g_2,p,s_2)$ 
in the variable $s_2$ with integer $p>1$. Moreover, this function is holomorphic in the entire complex plane because $1/\Gamma(s_2)$ is holomorphic, and we can speak of ``double L-values'' $L(g_1,g_2,p_1,p_2)$ with integers $p_1,p_2$ whenever $p_1>0$. Notice also that $L(g_1,g_2,p_1,p_2 )=0$ if $p_2 \le 0$ as one can see from~\eqref{2.3} since $\Gamma(s_2)$ has poles at nonpositive integers.

\begin{proposition}\label{doubL} Let $g_1, g_2$ be two modular forms of weight $k$ on a congruence subgroup, $g_1$ vanishing at $\infty$. Then for any integers $0< p_1 \le k$ and $p_2 < k$ one has 
\be{2.2}
\underset{q \to 1} \lim \;  D^{-p_2}{\Bigl(g_2 \cdot D^{-p_1} g_1 \Bigr)}(q) \=
L(g_1,g_2,p_1,p_2)
\ee
\end{proposition}
\begin{proof} Whenever $\re \, (p_2 + w) > k$ one has
\[\bal
\frac{\Gamma(w)}{(2\pi)^w} &L(g_1,g_2,p_1,p_2+w) \= \frac{\Gamma(w)}{(2\pi)^w} \sum_{n=1}^{\infty} \sum_{m=0}^{\infty} \frac{a_n
b_m}{n^{p_1}(n+m)^{p_2+w}} \\
&\= \sum_{n=1}^{\infty} \sum_{m=0}^{\infty} \frac{a_n
b_m}{n^{p_1}(n+m)^{p_2}} \int_0^{\infty} t^{w-1} e^{-2\pi t (n+m)}\\
&\= \int_0^{\infty} t^{w-1}
D^{-p_2}{\Bigl(g_2 \cdot D^{-p_1} g_1 \Bigr)} (it) dt \,.
\eal\]
By Mellin's inversion theorem with an arbitrary real $c > k-p_2$ one has   
\[\bal
&D^{-p_2}{\Bigl(g_2 \cdot D^{-p_1} g_1 \Bigr)} (it) \= \frac1{2\pi i} \int_{c-i\infty}^{c+i\infty}
\frac{\Gamma(w)}{(2\pi t)^w}
L(g_1,g_2,p_1,p_2+w) dw \\
&\;\= L(g_1,g_2,p_1,p_2) \+ \frac1{2\pi i} \int_{-\varepsilon-i\infty}^{-\varepsilon+i\infty}  \frac{\Gamma(w)}{(2\pi t)^w} L(g_1,g_2,p_1,p_2+w) dw  
\eal\]
with any $0 < \varepsilon < 1$ and we moved the path using the fact that $L(g_1,g_2,p_1,p_2+w)$ is everywhere holomorphic in $w$. The last integral obviously vanishes when $t\to 0$, and~\eqref{2.2} follows.
\end{proof}

In the case of Theorem~\ref{Thm1} we use the modular
parametrization~\eqref{L2mpar} and $h(t)=\frac1{1-t}$. Since $t=\infty$ at
$z=\frac12$ we consider the shifted forms 
\be{thm1gs}\bal
g_1(z) & \= \Bigl( \frac{Dt}{t} f \Bigr)(z+\frac12) \= 1+q-5q^2+q^3+11
q^4-24q^5+\dots \\
& \= E_{3,\chi_{-3}}(z) - 2 E_{3,\chi_{-3}}(2z) - 8 E_{3,\chi_{-3}}(4z) \\
g_2(z) & \= \Bigl( \frac{Dt}{1-t} f \Bigr)(z+\frac12) \= -q-4q^2-q^3+16
q^4+24q^5-4q^6+\dots \\
& \= -E_{3,\chi_{-3}}(z) - 7 E_{3,\chi_{-3}}(2z) + 8 E_{3,\chi_{-3}}(4z) \\
\eal\ee
where $E_{3,\chi_{-3}}$ is the Eisenstein series defined in~\eqref{Eis3}. The form
$g_1$ already appeared in~\eqref{thm1g1} and we had that $m(P_2)=-L(g_1,1)$.
Using Proposition~\ref{doubL} we now rewrite the statement of Theorem~\ref{Thm1}
as follows.

\begin{corollary}\label{cor1} With the modular forms $g_1,g_2$ of weight 3
defined in~\eqref{thm1gs} one has
\[
m(P_3) \- \frac34 m(P_2) \= - \frac6{\pi^2} L(g_2,g_1,2,1) \,. 
\]
\end{corollary}

Plugging in the values of $m(P_2)$ and $m(P_3)$ which we compute from (\ref{2var}), (\ref{3var})
into the formula given in Corollary \ref{cor1} we obtain the following relation
between double and ordinary L-values:
\be{L_relation}
L(g_2,g_1,2,1) = \frac{3\sqrt{3}\pi}{2^4} L(\ch3,2) - \frac{7}{6} \zeta(3). 
\ee{}
We give a straightforward proof of this relation in the next section.

\medskip
\medskip

For the Theorem~\ref{Thm2} we use the modular parametrization~\eqref{L3mpar} and
we have to consider two solutions with $h(t)=\frac1{1-t}$ and $h(t)=\frac{212
t^2 + 251t - 13}{(1-t)^3}$. Also we have $t=0$ at $z=i \infty$ and $t=\infty$ at
$z=0$. According to our strategy, we define the modular forms of weight~4
\[\bal
g_1 & \= \frac{Dt}{t} f \=  1+2q-14q^2+38q^3-142q^4+252q^5-266q^6+\dots, \\
g_2 & \= \frac{Dt}{1-t} f \= -q-7q^2-6q^3+5 q^4+120 q^5 +498 q^6 + \dots, \\
g_3 & \= \frac{212 t^2 + 251t - 13}{(1-t)^3} Dt \cdot f \=
13q+316q^2+2328q^3+\dots \\
\eal\]
(observe that they are the same ones as in~\eqref{thm2gs}). Here $g_1$ is a
holomorphic modular form, which already appeared in~\eqref{thm2g1} where we found
$m(P_3)=-L(g_1,1)$. The forms $g_2$ and $g_3$ are meromorphic with the poles at
the points where $t=1$. Using the fact that $t$ has no poles on the imaginary
half-axis we defined the corresponding double L-values  $L(g_2,g_1,3,1)$,
$L(g_3,g_1,3,1)$ in~\eqref{thm2DLs} in Section~\ref{sec:intro}.

\begin{corollary}\label{cor2} With the double L-values defined
in~\eqref{thm2DLs} one has
\[
m(P_4) \- \frac45 m(P_3) \= \frac{3\sqrt{5}\Omega_{15}^2}{10 \pi} L(g_3,g_1,3,1)
\-  \frac{3\sqrt{5}}{5 \pi^3 \Omega_{15}^2} L(g_2,g_1,3,1)\,. 
\]
\end{corollary}
\begin{proof} Due to Theorem~\ref{Thm2} we have that $m(P_4) = - \re \; \bigl( t
\frac{d}{dt}\bigr)^{-1} \, b(t) \Big|_{t = \infty}$. We know from
Proposition~\ref{L3mparPr} that this differential equation has modular
parametrization by $t(z)$ and $f(z)$. Therefore
\[
b(t(z)) \= \frac45 f(z) \+ f(z) \, D^{-3} \, \Bigl( -
\frac{3\sqrt{5}\Omega_{15}^2}{10 \pi} \, g_3(z) \+ \frac{3\sqrt{5}}{5 \pi^3
\Omega_{15}^2} \, g_2(z) \Bigr)\,.  
\]
We use the path in the upper halfplane from $z=i \infty$ to $z=0$ along
the imaginary half-axis, exactly where one has $-\infty < t(z) < 0$. As was explained in Section~\ref{sec:intro}, all three
forms are holomorphic along this path. For $g_j(z)$ with both $j=2,3$ we then
have  
\[
D^{-1} \Bigl( g_1 \cdot D^{-3} g_j \Bigr) (i v) \= (2 \pi)^4 \,
\int_{v}^{\infty} g_1(i s) \int_{s}^{\infty}\int_{s1}^{\infty}\int_{s2}^{\infty}
g_j(i s_3) \, ds_3 \, ds_2 \, ds_1 \, ds   
\]
and therefore the numbers~\eqref{thm2DLs} are the limiting values at $v=0$.
\end{proof}

\section{Explicit computation of the double L-value
in formula \eqref{L_relation}}

In this section we show how to compute the iterated integral
\begin{equation}\label{int}\bal
\int_0^{i \infty} g_1(z) \cdot D^{-2} g_2(z) dz &=
\frac1{2 \pi i} \int_1^{0} g_1(q) \cdot D^{-2} g_2(q) \frac{dq}{q} \\
&=- \frac1{2 \pi i} D^{-1} (g_1 \cdot D^{-2} g_2) \Big|_{q=1} \\
\eal\end{equation}
for the two modular forms $g_1$,$g_2$ of weight $3$ defined in (\ref{thm1gs})
which leads to an alternative proof of formula \eqref{L_relation}.

We use a powerful method due to Wadim Zudilin \cite{Z2}, \cite{Z3} of computing double
$L$-values of Eisenstein-like series. 
We are grateful to Wadim for explaining to us his method and its applicability in this situation.
Unfortunately, the more complicated $L$-values from Corrollary \ref{cor2}
do not seem to be computable in the same way due to the lack of an Eisenstein-like
representation for the forms $g_1$,$g_2$,$g_3$ refered to in
Corollary \ref{cor2}.

We briefly describe the method as follows: the Atkin-Lehner involution
$z \to -\frac1{12z}$ is applied to $g_1$, the resulting modular form being denoted by
$\hat{g}_1(z)$:
\[
g_1(z) = const \cdot \hat{g}_1(-\frac1{12z}) z^{-3},
\]
so that the integral \eqref{int} will take the following form:
\[
const \cdot \int_0^{i\infty} \hat{g}_1(-\frac1{12z}) D^{-2} g_2(z) z^{-3} dz.
\]
We then expand the integral as a quadruple sum, make a variable change and 
collapse the sum back in order to get
\[
const \cdot \int_0^{i\infty} f_1(u) (const + f_2(-\frac1{12u})) u\,du\\,
\]
where $f_1$ and $f_2$ are Eisenstein series of weight $1$.
We apply the Atkin-Lehner involution again:
\[ 
f_2(-\frac1{12u}) = const \cdot \hat{f}_2(u) u\,,
\]
this time rewriting the integral as
\[
const \cdot \int_0^{i\infty} f_1(u) (const + \hat{f}_2(u) \, u) u\,du\\ = const \cdot L(f_1,2) + const \cdot L(f_1 \hat{f}_2, 3).
\]
Here $f_1$ is an Eisenstein series of weight $1$ and character $\ch3$, and $L(f_1,2)$ will give the term with $\zeta(2) \cdot L(\ch3,2)$ in the final formula. The form $f_1 \hat{f}_2$ is an Eisenstein series for $SL_2(\mathbb{Z})$ of weight $2$, and $L(f_1 \hat{f}_2, 3)$ will give us the term with $\zeta(2) \cdot \zeta(3)$.

\medskip

Note that no regularization is necessary in our integral, since $g_2$ vanishes at $z=\infty$ and $g_1$ vanishes at $z=0$.

\medskip

In the course of the computation we will need to apply an Atkin-Lehner involution to Eisenstein series.
For that we express Eisenstein series as linear combinations of eta-products.
Let $N \ge 1$. Consider an eta-product
\[
f(z) = \prod_j \eta(d_j \cdot z)^{k_j} 
\]
where $k_j$ are integers and $d_j$ non-negative integers dividing $N$.
Let $d_j' = N / d_j$ and 
\[
\hat{f}(z) = \prod_j \eta(d_j' \cdot z)^{k_j}.  
\]

We have
\be{eta}
f(-\frac1{Nz}) = (-i)^w z^w \prod_j d_j'^{k_j/2} \hat{f}(z)
\ee
for $w = \frac12\sum_j k_j$ (the weight).
This follows from the basic transformation formula
\[
\eta\bigl(-\frac1{z}\bigr) \= \sqrt{-iz} \; \eta(z).
\]

We use the following two Eisenstein series of weight $3$:
\[\bal
\E3(z) &= -\frac19 \frac{\eta(z)^9}{\eta(3z)^3} = -\frac19 + \sum_{n,m \ge 1} \ch3(n)n^2 q^{nm} = -\frac19 + q - 3q^2 + q^3 + \dots\,, \\
\EE3(z) &= \frac{\eta(3z)^9}{\eta(z)^3} = \sum_{n,m \ge 1} \ch3(m)n^2 q^{nm} = q + 3q^2 + 9q^3 + 13q^4 + \dots\,. \\
\eal\]
Then
\[\bal
g_1(z) &= (1,-2,-8) \cdot (\E3(z),\E3(2z),\E3(4z))^t\,, \\
g_2(z) &= (-1,-7,8) \cdot (\E3(z),\E3(2z),\E3(4z))^t\,,
\eal\]
where $(a,b,c)\cdot(d,e,f)^t=ad+be+cf$. Application of~\eqref{eta} gives 
\[\bal
\E3\Bigl(-\frac1{12z}\Bigr) &\= -2^6 3^{5/2} i z^3 \EE3(4z)\,, \\
\E3\Bigl(-\frac2{12z}\Bigr) &\= \E3(-\frac1{6z}) = -2^3 3^{5/2} i z^3 \EE3(2z)\,,\\
\E3\Bigl(-\frac4{12z}\Bigr) &= \E3(-\frac1{3z}) = -3^{5/2} i z^3 \EE3(z)\,, \\
\eal\]
and hence
\[
(a,b,c) \cdot (\E3(z),\E3(2z),\E3(4z))^t (-\frac1{12z}) = -3^{5/2} i z^3 (c,2^3b,2^6a) \cdot (\EE3(z),\EE3(2z),\EE3(4z))^t.
\]
In particular, we have
\[\bal
g_1(-\frac1{12z}) &= 8 \cdot 3^{5/2} iz^3 \hat{g}_1(z) \\
\hat{g}_1(z) &= (1,2,-8) \cdot (\EE3(z),\EE3(2z),\EE3(4z))^t\,, \\ 
&= q + 5 q^2 + 9 q^3 + 11 q^4 + 24 q^5 + \dots
\eal\]
or, equivalently,
\begin{equation}\label{g1}
g_1(z) = - \frac{i  z^{-3} }{2^3 3^{1/2}} \hat{g}_1(-\frac1{12z})\,.
\end{equation}
Formula (\ref{g1}) allows us to rewrite the iterated integral (\ref{int}) as
\begin{equation}\label{int2}
-\frac{i}{2^3 3^{1/2}}\int_0^{i\infty} \hat{g}_1(-\frac1{12z}) D^{-2} g_2(z) z^{-3} dz \,.\\ 
\end{equation}

We now make use of quadruple sums. For that we write our form $\hat{g}_1$ and $g_2$ as
\[\bal
\hat{g}_1(z) &= \sum_{m_1,n_1\ge 1} a_1(m_1) b_1(n_1) n_1^2 q^{m_1 n_1} = q + 5q^2 + 9q^3 + 11q^4 + 24q^5 + \dots\,,\\ 
g_2(z) &= \sum_{m_2,n_2\ge 1} a_2(m_2) b_2(n_2) n_2^2 q^{m_2 n_2} = -q - 4q^2 - q^3 + 16q^4 + 24q^5 + \dots \,,\\ 
\eal\]
where
\[\bal
a_1(m) &= \ch3(m) \,,\\
b_1(n) &= 1 + \frac12 [n \; \text{even}] - \frac12 [n \; \text{divisible by 4}] \,,\ \\ 
a_2(m) &= -1 -7 [m \;\text{even}] + 8 [m \; \text{divisible by 4}] \,,\\\
b_2(n) &= \ch3(n) \,,\\\ 
\eal\]
and $[\dots]$ means $1$ when the respective condition is satisfied and $0$ otherwise. By using the expansions
\[\bal
\hat{g}_1(-\frac1{12z}) &\= \sum_{m_1,n_1\ge 1} a_1(m_1) b_1(n_1) n_1^2 \exp\Bigl(-\frac{2 \pi i n_1 m_1}{12z}\Bigr)\,, \\
D^{-2} g_2(z) &\= \sum_{m_2,n_2\ge 1} a_2(m_2) b_2(n_2) \frac1{m_2^2} \exp\Bigl(2 \pi i m_2 n_2 z\Bigr) \\ 
\eal\]
in~\eqref{int2}, we obtain a quadruple sum:
\[\bal
& -\frac{i}{2^3 3^{1/2}}\int_0^{i\infty} \hat{g}_1(-\frac1{12z}) D^{-2} g_2(z) z^{-3} dz \\
=& -\frac{i}{2^3 3^{1/2}}\sum_{m_1,n_1,m_2,n_2} a_1(m_1) b_1(n_1) a_2(m_2) b_2(n_2) \frac{n_1^2}{m_2^2} \int_0^{i\infty} exp\Bigl(2\pi i (-\frac{m_1 n_1}{12z} + m_2 n_2 z) \Bigr) 
z^{-3} dz. \\ 
\eal\]
Now we change variable in the integral. First, we let $w = -\frac1{12z}$ and obtain
\[
 -\frac{12^2i}{2^3 3^{1/2}} \sum_{m_1,n_1,m_2,n_2} a_1(m_1) b_1(n_1) a_2(m_2) b_2(n_2) \frac{n_1^2}{m_2^2} \int_0^{i\infty} 
exp\Bigl(2\pi i (m_1 n_1 w - \frac{m_2 n_2}{12w}) \Bigr) w \, dw \;.
\]
With $u = \dfrac{n_1 w}{m_2}$ we  get
\[\bal
 -2 \cdot 3^{3/2} i\sum_{m_1,n_1,m_2,n_2} a_1(m_1) b_1(n_1) a_2(m_2) b_2(n_2) \int_0^{i\infty} exp\Bigl(2\pi i (m_1 m_2 u - \frac{n_1 n_2}{12u}) \Bigr) 
u \, du \;,
\eal\]
which we rewrite as
\begin{equation}\label{int3}
-2 \cdot 3^{3/2} i \int_0^{i\infty} f_1(u) \Bigl(f_2\bigl(-\frac1{12u}\bigr) - \frac16\Bigr) u\,du \;,\\
\end{equation}
where
\[\bal
f_1(z) &= \sum_{m_1,m_2\ge 1} a_1(m_1) a_2(m_2) q^{m_1 m_2} \;\= -q - 7q^2 - q^3 + 7q^4 - 7q^6 - 2q^7 + \dots \,,\\ 
f_2(z) &= \frac16 + \sum_{n_1,n_2\ge 1} b_1(n_1) b_2(n_2) q^{n_1 n_2} \= \frac16 + q + 1/2 q^2 + q^3 + 1/2 q^4 + 1/2q^6 + 2q^7 + \dots  \,. \\ 
\eal\]
(The term $\frac16$ turns $f_2(z)$ into a modular form.)

Consider the Eisenstein series of weight~1
\[\bal
E_1(z) &= \frac16 \+ \sum_{m,n \ge 1} \ch3(m) q^{nm}\,. \\
\eal\]
Then for any positive integer $l$ we have 
\[\bal
\sum_{m,n \ge 1} \ch3(m)\, [n \; \text{divisible  by} \; l] \, q^{nm} &\= E_1(l\,z) - \frac16,
\eal\]
and therefore
\[\bal
f_1(z) &\= -E_1(z) \- 7 E_1(2z) \+ 8 E_1(4z)\,, \\
f_2(z) &\= E_1(z) \+ \frac12 E_1(2z) \- \frac12 E_1(4z) \,.\\
\eal\]
In order to evaluate $(\ref{int3})$ we apply an Atkin-Lehner involution to $f_2$.
We first express $f_2$ as a combination of eta-products:
\[
f_2(z) = \frac12 \frac{\eta(4z)^2\eta(12z)^2}{\eta(2z)\eta(6z)} + \frac16 \frac{\eta(2z)^6 \eta(3z)}{\eta(z)^3 \eta(6z)^2}\,, \\
\]
then we use (\ref{eta}):
\[\bal
f_2\bigl(-\frac1{12z}\bigr) &= \frac12 \Bigl(\frac{3^2}{6\cdot 2}\Bigr)^{1/2}  (-iz) \frac{\eta(3z)^2\eta(z)^2}{\eta(6z)\eta(2z)} +
\frac16 \Bigl(\frac{6^6 4}{12^3 2^2}\Bigr)^{1/2}  (-iz) \frac{\eta(6z)^6 \eta(4z)}{\eta(12z)^3 \eta(2z)^2} \\
&= -\frac{3^{1/2}}{2} iz \hat{f}_2(z), \\
\hat{f}_2(z) &= \frac12 \frac{\eta(3z)^2\eta(z)^2}{\eta(6z)\eta(2z)} + \frac{\eta(6z)^6 \eta(4z)}{\eta(12z)^3 \eta(2z)^2} 
= - E_1(z) + 2 E_1(2z) + 8 E_1(4z) \\
&= 3/2 - q + 2 q^2 - q^3 + 7q^4 + 2q^6 + \dots \,.
\eal\]
For a modular form $f$ we have 
\[
 \int_0^{i \infty} f(z) z^{k-1} dz  \= \frac{(k-1)!}{(-2 \pi i)^k} \, L(f,k)\,.
\]
Using this fact, we continue rewriting the integral (\ref{int3}):
\[\bal
\dots &\= 3^{1/2} i \int_0^{i \infty} f_1(u) \, u \, du - 3^2  \int_0^{i \infty} f_1(u) \hat{f}_2(u) \, u^2\, du \\
&\= \frac{3^{1/2} i}{(-2\pi i)^2} \, L(f_1,2) \- \frac{3^2 \cdot 2}{(-2\pi i)^3} \, L(f_1 \hat{f}_2, 3) \\
&\= -\frac{3^{1/2} i}{4 \pi^2} \, L(f_1,2) \+ \frac{3^2 i}{2^2 \pi^3} \, L(f_1 \hat{f}_2, 3) \,.\\
\eal\]
Now we evaluate the $L$-values that have appeared. Since $f_1(z) = - E_1(z) - 7 E_1(2z) + 8 E_1(4z)$, we have $L(f_1, s) = (-1 - 7 \cdot 2^{-s} + 8 \cdot 4^{-s}) \zeta(s) L(\ch3,s)$ and 
\[
L(f_1, 2) \= \bigl(-1 - \frac{7}{2^2} + \frac{8}{4^2}\bigr) \zeta(2) L(\ch3,2) \= -\frac{3}{8} \pi^2 \, L(\ch3,2)\,.
\]
Using the representation
\[
f_1(z) \hat{f}_2(z) \= -\frac32 G_2(z) - 5 G_2(2z) + \frac{19}{2} G_2(3z) + 24 G_2(4z) - 35 G_2(6z) + 8 G_2(12z)\,, 
\]
we have
\[\bal
L(f_1 \hat{f}_2\,,\, s) &\= \bigl(-\frac32 - 5 \cdot 2^{-s} + \frac{19}{2} \cdot 3^{-s} + 24 \cdot 4^{-s} - 35 \cdot 6^{-s} + 8 \cdot 12^{-s}\bigr) \, \zeta(s) \, \zeta(s-1) \\
L(f_1 \hat{f}_2\,,\, 3) &\= \bigl(-\frac32 - \frac{5}{2^3} + \frac{19}{2\cdot3^3} + \frac{24}{4^3} - \frac{35}{6^3} + \frac{8}{12^3}\bigr) \,\zeta(3) \, \zeta(2) \\
&\= -\frac{14}{9} \,\zeta(3)\, \zeta(2) \= -\frac{7 \, \pi^2}{27} \, \zeta(3) \\
\eal\]
Eventually we finish the computation of the integrals in (\ref{int}):
\[
\int_0^{i \infty} g_1(z) \cdot D^{-2} g_2(z) dz \= \frac{3^{3/2} \, i}{2^5} \, L(\ch3,2) \- \frac{7 \, i}{12\pi} \,\zeta(3)\,. 
\]
Multiplying by $-2\pi i$ we get
\[
D^{-1} \Bigl(g_1(q) \, D^{-2}g_2  \,(q) \Bigr) \Big|_{q=1} \= \frac{3\sqrt{3}\, \pi}{2^4} \, L(\ch3,2) \- \frac{7}{6} \, \zeta(3) \,,
\]
which is the same as (\ref{L_relation}).


\begin{thebibliography}{xxx}
\bibitem[BLVD]{BLVD} D. Boyd, D. Lind, F. Rodriguez-Villegas, C. Deninger, 
\emph{The Many Aspects of Mahler's Measure}, 
final report of 2003 Banff workshop
{\tt http://www.birs.ca/workshops/2003/03w5035/report03w5035.pdf}

\bibitem[PTV]{PTV} C.Peters, J. Top, M. van der Vlugt
\emph{The Hasse zeta function of a K3 surface related to the number of words of weight 5 in the Melas codes},
J. reine angew. Math. 432 (1992), 151-176 

\bibitem[RV]{MMM} F. Rodriguez-Villegas
\emph{Modular Mahler Measures. I}
Topics in number theory (University Park, PA, 1997), 17--48, 
Math. Appl., 467, Kluwer Acad. Publ., Dordrecht, 1999

\bibitem[RTV]{VTV} F. Rodriguez-Villegas, R. Toledano, J. D. Vaaler, 
\emph{Estimates for Mahler's measure of a linear form}
Proc. Edinb. Math. Soc. 47 (2004), 473--494

\bibitem[SB]{SB} J. Stienstra, F. Beukers
\emph{On the Picard-Fuchs equation and the formal Brauer group of certain elliptic K3-surfaces}
Mathematische Annalen 271 (1985) pp.269--304.

\bibitem[Sm]{Sm1} C. J. Smyth
\emph{On measures of polynomials in several variables}
Bull. Austral. Math. Soc. 23 (1981) 49--63

\bibitem[Sr]{Ramesh} R. Sreekantan
\emph{Values of Multiple L-functions and Periods of Integrals}
{\tt http://www.isibang.ac.in/\textasciitilde rsreekantan/MLV.pdf}

\bibitem[Ver96]{Verrill} H.A.Verrill
\emph{Root lattices and Pencils of Varieties}
J. Math. Kyoto Univ. 36 (1996), no. 2, 423--446

\bibitem[Ver08]{Ver1} H.A.Verrill,
\emph{Sums of squares of binomial coefficients, with applications to Picard-Fuchs equations},
{\tt arXiv:math.CO/0407327}

\bibitem[Yan]{Yang} Y. Yang
\emph{Ap\'{e}ry limits and special values of L-functions}
J.Math. Anal. Appl. 343 (2008) 492--513

\bibitem[Zag]{Zag} D.Zagier
\emph{Elliptic modular forms and their applications}
The 1-2-3 of modular forms, 1--103, Universitext, Springer, Berlin, 2008

\bibitem[Zud1]{Zud} W.Zudilin,
\emph{Arithmetic gypergeometric series},
Russian Math. Surveys 66:2 (2011), 1-51 

\bibitem[Zud2]{Z2} W.Zudilin,
\emph{Transformations of $L$-values},
{\tt arXiv:1202.5630v1 [math.NT]}

\bibitem[Zud3]{Z3} W.Zudilin,
\emph{Perioddness of $L$-values}, 
preprint of the Max Planck Institute f\"ur Mathematick, available at {\tt http://www.mpim-bonn.mpg.de/preblob/5314}

\end{thebibliography}
\end{document}